\documentclass[11pt]{article}

\usepackage{amsmath,epsfig,amssymb,amsbsy,verbatim,array,color,graphics,graphicx}
\usepackage{amssymb,amsmath,amsthm}
\usepackage{graphicx}
\usepackage{subdepth}

\newtheorem{theorem}{Theorem}
\newtheorem{lemma}[theorem]{Lemma}
\newtheorem{prop}[theorem]{Proposition}
\newtheorem{corollary}[theorem]{Corollary}

\newtheorem{obs}[theorem]{Observation}

\definecolor{darkblue}{rgb}{0,0,0.7}
\definecolor{darkgreen}{rgb}{0,0.3,0}
\definecolor{darkred}{rgb}{0.7,0,0}

\def\dcup{\,\dot\cup\,}
\def\dotmns{-^{\hspace{-0.19cm}.}\,}
\def\wmns{-^{\hspace{-0.24cm}{\textup{\tiny $w$}}}\,}

\newdimen\unit\newdimen\psep\newcount\nd\newcount\ndx\newbox\dotb\newbox\ptbox
\newdimen\dx\newdimen\dy\newdimen\dxx\newdimen\dyy\newdimen\hgt
\newdimen\xoff\newdimen\yoff
\newcommand\clap[1]{\hbox to 0pt{\hss{#1}\hss}}
\newcommand\vdisk[1]{{\font\dotf=cmr10 scaled #1\dotf.}}
\newcommand\varline[2]{\setbox\dotb\hbox{\vdisk{#1}}\xoff=-.5\wd\dotb
\wd\dotb=0pt\yoff=-.5\ht\dotb\psep=#2\ht\dotb}
\newcommand\varpt[1]{\setbox\ptbox\clap{\vdisk{#1}}\setbox\ptbox
\hbox{\raise-.5\ht\ptbox\box\ptbox}}
\newcommand\cpt{\copy\ptbox}
\newcommand\point[3]{\rlap{\kern#1\unit\raise#2\unit\hbox{#3}}}
\newcommand\setnd[4]{\dx=#3\unit\advance\dx-#1\unit\divide\dx by\psep
\dy=#4\unit\advance\dy-#2\unit\divide\dy by\psep \multiply\dx
by\dx\multiply\dy by\dy\advance\dx\dy\nd=1\advance\dx-1sp
\loop\ifnum\dx>0\advance\dx-\nd sp\advance\nd1\advance\dx-\nd
sp\repeat}

\newcommand\dline[5]{{\nd=#5\hgt=#2\unit\dx=#3\unit\advance\dx-#1\unit
\divide\dx by\nd\dy=#4\unit\advance\dy-#2\unit\divide\dy by\nd
\advance\hgt\yoff\rlap{\kern#1\unit\kern\xoff\loop\ifnum\nd>1\advance\nd-1
\advance\hgt\dy\kern\dx\raise\hgt\copy\dotb\repeat}}}

\newcommand\qellip[4]{{\setnd{0}{0}{#3}{#4}\dx=\unit\dy=0pt\raise\yoff\rlap{%
\kern#1\unit\kern\xoff\raise#2\unit\hbox{\loop\ifnum\dx>0\rlap{\kern#3\dx
\raise#4\dy\copy\dotb}\hgt=\dx\divide\hgt
by\nd\advance\dy\hgt\hgt=\dy \divide\hgt
by\nd\advance\dx-\hgt\repeat\rlap{\raise#4\dy\copy\dotb}}}}}
\newcommand\bez[6]{{\setnd{#1}{#2}{#3}{#4}\ndx=\nd\setnd{#3}{#4}{#5}{#6}
\ifnum\ndx>\nd\nd=\ndx\fi\dx=#3\unit\advance\dx-#1\unit\dy=#4\unit
\advance\dy-#2\unit\dxx=#5\unit\advance\dxx-#1\unit\dyy=#6\unit\advance
\dyy-#2\unit\advance\dxx-2\dx\advance\dyy-2\dy\divide\dxx
by\nd\divide\dyy
by\nd\advance\dx.25\dxx\advance\dy.25\dyy\divide\dx
by\nd\divide\dy by\nd \multiply\nd
by2\dx=100\dx\dy=100\dy\dxx=100\dxx\dyy=100\dyy\divide\dxx by\nd
\divide\dyy
by\nd\hgt=#2\unit\raise\yoff\rlap{\kern#1\unit\kern\xoff
\raise\hgt\copy\dotb\loop\ifnum\nd>0\advance\nd-1\advance\hgt0.01\dy
\kern0.01\dx\raise\hgt\copy\dotb\advance\dx\dxx\advance\dy\dyy\repeat}}}
\newcommand\ptu[3]{\point{#1}{#2}{\cpt\raise1ex\clap{$\scriptstyle{#3}$}}}
\newcommand\ptd[3]{\point{#1}{#2}{\cpt\raise-1.8ex\clap{$\scriptstyle{#3}$}}}
\newcommand\ptr[3]{\point{#1}{#2}{\cpt\raise-.4ex\rlap{$\ \scriptstyle{#3}$}}}
\newcommand\ptl[3]{\point{#1}{#2}{\cpt\raise-.4ex\llap{$\scriptstyle{#3}\ $}}}
\newcommand\ptlu[3]{\point{#1}{#2}{\raise.8ex\clap{$\scriptstyle{#3}$}}}
\newcommand\ptld[3]{\point{#1}{#2}{\raise-1.6ex\clap{$\scriptstyle{#3}$}}}
\newcommand\ptlr[3]{\point{#1}{#2}{\raise-.4ex\rlap{$\,\scriptstyle{#3}$}}}
\newcommand\ptll[3]{\point{#1}{#2}{\raise-.4ex\llap{$\scriptstyle{#3}\,$}}}

\newcommand\thnline{\varline{400}{.6}}

\varpt{2500}\thnline\unit=2em

\begin{document}

\title{Cops and Robber on Cartesian products\\
 and some classes of hypergraphs\\[2ex]}
\author{Pinkaew Siriwong\footnotemark[2]
\and Ratinan Boonklurb\footnotemark[2]
\and Henry Liu\thanks{Corresponding author
\newline\indent\hspace{0.12cm}$^\dag$Department of Mathematics and Computer Science, Chulalongkorn University, Bangkok 10300, Thailand. E-mail: psiriwong@yahoo.com$\,\mid\,$ratinan.b@chula.ac.th
\newline\indent\hspace{0.12cm}$^\ddag$School of Mathematics, Sun Yat-sen University, Guangzhou 510275, China. E-mail: liaozhx5@mail.sysu.edu.cn
\newline\indent\hspace{0.12cm}$^\S$Department of Mathematics, Ramkhamhaeng University, Bangkok 10240, Thailand. E-mail: sin\_sirirat@ru.ac.th
}
\footnotemark[3]
\and Sirirat Singhun\footnotemark[4]
\\[2ex]
}
\date{24 July 2021}
\maketitle
\begin{abstract}
The game of \emph{Cops and Robber} is a pursuit-evasion game which is usually played on a connected graph. In the game, a set of cops and a robber move around the vertices of a graph along edges, where the cops aim to capture the robber, while the robber aims to avoid capture. Much research about this game have been done since the early 1980s. The game has a natural generalisation to being played on connected hypergraphs, where the cops and the robber may now move along hyperedges. In this paper, we shall provide a characterisation of all hypergraphs where one cop is sufficient to capture the robber. The \emph{cop-number} of a connected hypergraph is the minimum number of cops required in order to capture the robber. We shall prove some results about the cop-number of certain hypergraphs, including hypertrees and Cartesian products of hypergraphs.\\

\noindent\textbf{AMS Subject Classificiation (2020):} 05C57, 05C65\\

\end{abstract}

\section{Introduction}

In this paper, all graphs and hypergraphs are finite, undirected, and without multiple edges and loops. Our notations for graphs and hypergraphs throughout are fairly standard. For a graph $G$ and $U\subset V(G)$, we use $G-U$ to denote the subgraph of $G$ induced by $V(G)\setminus U$. In particular, if $x\in V(G)$, we write $G-x$ instead of $G-\{x\}$. Thus, $G-x$ is obtained by deleting $x$ and all edges incident to $x$ from $G$. If $G$ is a connected graph, then $x$ is a \emph{cut-vertex} of $G$ if $G-x$ is a disconnected graph, so that $G-x$ has at least two components. The minimum degree of $G$ is denoted by $\delta(G)$. For a hypergraph $\mathcal H$, the \emph{rank} $r(\mathcal H)$ and \emph{anti-rank} $s(\mathcal H)$ of $\mathcal H$ are the maximum and minimum size of an edge of $\mathcal H$. If $r(\mathcal H)=s(\mathcal H)=r$, then $\mathcal H$ is \emph{$r$-uniform}, or simply \emph{uniform} if we do not wish to make reference to $r$. Two vertices $x,y\in V(\mathcal H)$ are \emph{adjacent} if $x$ and $y$ both belong to some edge of $\mathcal H$, in which case $y$ is a \emph{neighbour} of $x$. The \emph{open neighbourhood} $N_{\mathcal H}(x)$ of $x$ is the set of all neighbours of $x$ in $\mathcal H$, and the \emph{closed neighbourhood} of $x$ is $N_{\mathcal H}[x]=N_{\mathcal H}(x)\cup\{x\}$. We write $N[x]$ if it is clear which hypergraph is in consideration. The \emph{$2$-section} of $\mathcal H$, denoted by $[\mathcal H]_2$, is the graph with vertex set $V([\mathcal H]_2)=V(\mathcal H)$, and edge set $E([\mathcal H]_2)=\{xy:x,y\in e$ for some $e\in E(\mathcal H)\}$. We say that $\mathcal H$ is \emph{connected} if $[\mathcal H]_2$ is connected. For any other undefined terms about graphs and hypergraphs, we refer to the books \cite{Bol98,Vol09}.

The game of \emph{Cops and Robber} is usually played on a finite connected graph $G$. Two players $C$ and $R$, whom we shall always assume are female and male, control a set of $k$ \emph{cops} and the \emph{robber}. The player $C$ starts the game by placing the $k$ cops at some vertices of $G$, then $R$ chooses a vertex to place the robber. $C$ then makes a \emph{move}, which means that she may move some (possibly none or all) of the $k$ cops from their present occupied vertices to adjacent vertices, and leave the remaining cops at their present vertices. $R$ then similarly either moves the robber or passes his turn. These two moves are then repeated alternately. At any stage of the game, including the very first step, $C$ may have more than one cop occupying the same vertex of $G$. If after a finite number of moves, $C$ manages to occupy the same vertex as the robber with one of the cops, i.e., the cop `captures' the robber, then she wins the game. We say that the graph $G$ is \emph{$k$-cop-win}, or simply \emph{cop-win} for $k=1$, if $C$ always has a winning strategy with $k$ cops. The \emph{cop number} $c(G)$ of $G$ is the minimum value of $k$ such that $G$ is a $k$-cop-win graph.

The version of this game with one cop (thus often called \emph{Cop and Robber}) was introduced independently by Quilliot \cite{Qui78} in 1978, and Nowakowski and Winkler \cite{NW83} in 1983. Aigner and Fromme \cite{AF84} introduced the version with multiple cops, along with the concept of the cop number, in 1984. They proved that any planar graph has cop number at most $3$, and some other results about the cop number in relation to the minimum and maximum degrees of the graph. Since then, much research have been done on the topic of the game of Cops and Robber. Perhaps the most glaring open problem is Meyniel's conjecture, which was posed in 1985 and mentioned by Frankl in \cite{Fra87}. The conjecture asserts that for any connected graph $G$ on $n$ vertices, the cop number satisfies $c(G)=O(\sqrt{n})$. See \cite{BB12,LP12,SS11} for some progress on this conjecture. Several variants of Cops and Robber have also been considered. For example, the \emph{Fast Robber} variant allows the robber to move at a greater speed $s>1$. That is, on the player $R$'s turn, he may move the robber along any walk of length at most $s$ which does not contain any cops. See \cite{AM11,BBNS17,FGKNS10,FKL12,Meh11} for some research in this direction. There are also versions of the game where multiple robbers are present. For a general overview of the game of Cops and Robber on graphs, we refer the reader to the book by Bonato and Nowakowski \cite{BN11}.

We may consider playing the game of Cops and Robber on hypergraphs. The rules of the game are exactly analogous to the graphs setting, where now the game is played on a finite connected hypergraph $\mathcal H$. Two players $C$ and $R$, in that order, place a set of $k$ cops and a robber at some vertices of $\mathcal H$. Then $C$ makes a \emph{move}, which means that she moves some of the $k$ cops from their present vertices to adjacent vertices, and leave the remaining cops at their present vertices. $R$ then similarly either moves the robber or passes his turn, and then moves are made alternately between $C$ and $R$. More than one cop may occupy the same vertex at any stage of the game. Each pair of consecutive moves made by $C$ and then by $R$ is a \emph{round}, where the first round consists of the initial placements of the $k$ cops and the robber. As before, $C$ wins the game if she captures the robber with a cop. $\mathcal H$ is a \emph{$k$-cop-win} hypergraph, or \emph{cop-win} for $k=1$, if $C$ has a winning strategy with $k$ cops. The \emph{cop number} $c(\mathcal H)$ of $\mathcal H$ is the minimum value of $k$ such that $\mathcal H$ is a $k$-cop-win hypergraph. We remark that if $\mathcal H$ is a $k$-cop-win hypergraph, then $C$ always has a winning strategy no matter where she places the $k$ cops in the first round, although the number of rounds required to capture the robber may vary. Without danger of confusion, we may use $C_1,\dots,C_k$ to denote the $k$ cops, or simply $C$ if $k=1$, and use $R$ to denote the robber.

Two early references which demonstrated the notion of the game of Cops and Robber in relation to hypergraphs are due to Gottlob, Leone and Scarcelloc \cite{GLS03} from 2003, and Adler \cite{Adl04} from 2004. They studied how the hypertree width of a hypergraph is related to both the games of \emph{Marshals and Robber} (a variant game), and Cops and Robber. Baird \cite{Bai11} also considered the game of Cops and Robber on hypergraphs in 2011, when he studied the game on certain hyperpaths and hypercycles, and proved some other results.

Clearly for a hypergraph $\mathcal H$ and every $x\in V(\mathcal H)$, we have
\begin{equation}\label{nbrhdppty}
N_{\mathcal H}[x]=N_{[\mathcal H]_2}[x].
\end{equation}
We may easily deduce the following.

\begin{obs}\label{csameobs}
Let $\mathcal H$ be a connected hypergraph. Then $c(\mathcal H)=c([\mathcal H]_2)$.
\end{obs}

Indeed, if $c(\mathcal H)$ cops are played on $[\mathcal H]_2$, then in view of (\ref{nbrhdppty}), they have a winning strategy by playing as they would in $\mathcal H$, so that $c(\mathcal H)\ge c([\mathcal H]_2)$. Similarly, we may obtain $c(\mathcal H)\le c([\mathcal H]_2)$. We will see that Observation \ref{csameobs} will be very useful, as it will allow us to derive some results about the cop number of a hypergraph from their analogues for a graph. However, despite having Observation \ref{csameobs}, the study of the game of Cops and Robber on hypergraphs is far from trivial when one attempts to deduce results from analogous results for graphs. For many hypergraphs, for example hypertrees, the structure of their $2$-sections are far from simple.

There are several fundamental results about Cops and Robber on graphs. In this paper, we shall aim to derive analogues of these results for hypergraphs. Firstly, Nowakowski and\, Winkler\, \cite{NW83}\, gave\, a\, characterisation\, of\, cop-win\, graphs\, in\, terms\, of\, a\, structural\, property of graphs called \emph{dismantlable}. They proved that a graph is cop-win if and only if it is dismantlable. In Section \ref{charsect}, we shall derive an analogous characterisation for cop-win hypergraphs. Next, in Section \ref{specificsect}, we shall prove that all hypertrees are cop-win hypergraphs, and consider the cop number of complete multipartite hypergraphs. Finally, many results about the cop number of the Cartesian product of graphs have been proved. In Section \ref{prodsect}, we shall derive analogous results for the Cartesian product of hypergraphs.

\section{Characterisation of cop-win hypergraphs}\label{charsect}

In this section, we shall derive a characterisation of cop-win hypergraphs. For the case of graphs, such a characterisation was given by Nowakowski and Winkler \cite{NW83} in terms of a structural property of graphs called \emph{dismantlable}, as follows. For a graph $G$, a vertex $v\in V(G)$ is a \emph{corner} if $N[v]\subset N[u]$ for some $u\in V(G)\setminus\{v\}$, in which case $u$ is a \emph{cover} of $v$. A graph $G$ is \emph{dismantlable} if there exists an ordering $v_1,\dots,v_n$ of the vertices of $G$ such that for every $1\le i<n$, $v_i$ is a corner in the subgraph $G-\{v_1,\dots,v_{i-1}\}$. Nowakowski and Winkler proved the following result, which was also mentioned by Aigner and Fromme \cite{AF84}.

\begin{theorem}\label{NWthm}\textup{\cite{AF84,NW83}}
A connected graph is cop-win if and only if it is dismantlable.
\end{theorem}

We shall extend Theorem \ref{NWthm} to hypergraphs. To do this, we shall formulate the definition of a \emph{dismantlable hypergraph}. For a hypergraph $\mathcal H$, a vertex $v\in V(\mathcal H)$ is a \emph{corner} if $N[v]\subset N[u]$ for some $u\in V(\mathcal H)\setminus\{v\}$, in which case $u$ is a \emph{cover} of $v$. Note that the definition of a dismantlable graph is saying that we may delete vertices successively in such a way that at each step, the vertex to be deleted is a corner of the subgraph at the respective step. Thus, we consider how to delete a vertex $x$ from a hypergraph $\mathcal H$. There is more than one possibility to define such a deletion. We would not like to delete $x$ and all edges incident to $x$, since this may possibly yield a hypergraph $\mathcal H'$ which is disconnected if $\mathcal H$ is connected, even if $x$ is a corner of $\mathcal H$. This is undesirable, since we would like the Cops and Robber game to be playable on $\mathcal H'$. To overcome this difficulty, we shall retain the remaining parts of some of the edges incident with $x$ in the deletion process. For a hypergraph $\mathcal H$ and $x\in V(\mathcal H)$, we define the hypergraph $\mathcal H \dotmns x$ where
\begin{align*}
V(\mathcal H\dotmns x) &= V(\mathcal H) \setminus\{x\},\\
E(\mathcal H\dotmns x) &= \{e\in E(\mathcal H):x\not\in e\}\dcup\{f\setminus\{x\}: f\in E(\mathcal H),\,x\in f\textup{ and }|f|\ge 3\}.
\end{align*}

It is not hard to see that if $\mathcal H=G$ is a graph, then the above definition coincides with the usual definition of deleting $x$ from $G$. That is, $\mathcal H\dotmns x = G-x$. With this definition of the deletion of a vertex from a hypergraph, we do have the condition which says that the connectedness property is preserved when a corner is deleted.

\begin{prop}\label{connprp}
Let $\mathcal H$ be a connected hypergraph, and $x\in V(\mathcal H)$ be a corner. Then $\mathcal H\dotmns x$ is also connected.
\end{prop}

Before we prove Proposition \ref{connprp}, we prove two lemmas.

\begin{lemma}\label{lm1}
Let $\mathcal H$ be a hypergraph, and $x\in V(\mathcal H)$. Then $x$ is a corner of $\mathcal H$ if and only if $x$ is a corner of $[\mathcal H]_2$.
\end{lemma}

\begin{proof}
Suppose that $x$ is a corner of $\mathcal H$, with cover $y\in V(\mathcal H)\setminus\{x\}$. Then by (\ref{nbrhdppty}), $N_{[\mathcal H]_2}[x]=N_{\mathcal H}[x]\subset N_{\mathcal H}[y] = N_{[\mathcal H]_2}[y]$, so that $x$ is a corner of $[\mathcal H]_2$ with cover $y$.

Conversely, suppose that $x$ is a corner of $[\mathcal H]_2$, with cover $y\in V([\mathcal H]_2)\setminus\{x\}$. Again by (\ref{nbrhdppty}), $N_{\mathcal H}[x]=N_{[\mathcal H]_2}[x]\subset N_{[\mathcal H]_2}[y] = N_{\mathcal H}[y]$, so that $x$ is a corner of $\mathcal H$ with cover $y$.
\end{proof}

\begin{lemma}\label{lm2}
Let $\mathcal H$ be a hypergraph, and $x\in V(\mathcal H)$. Then $[\mathcal H\dotmns x]_2=[\mathcal H]_2-x$.
\end{lemma}

\begin{proof}
Obviously, both $[\mathcal H\dotmns x]_2$ and $[\mathcal H]_2-x$ have the vertex set $V(\mathcal H)\setminus\{x\}$. We show that  $[\mathcal H\dotmns x]_2$ and $[\mathcal H]_2-x$ have the same edge set.

Suppose firstly that $uv\in E([\mathcal H\dotmns x]_2)$. Then either $u,v\in e$ for some $e\in E(\mathcal H)$ with $x\not\in e$, or $u,v\in f\setminus\{x\}$ for some $f\in E(\mathcal H)$ with $x\in f$ and $|f|\ge 3$. We have either $e$ or $f\setminus\{x\}$ inducing a clique in $[\mathcal H]_2$ which contains $u$ and $v$ but not $x$, so that $uv\in E([\mathcal H]_2-x)$.

Conversely, suppose that $uv\in E([\mathcal H]_2-x)$. Then either $u,v\in e$ for some $e\in E(\mathcal H)$ with $x\not\in e$, or $u,v\in f$ for some $f\in E(\mathcal H)$ with $x\in f$. If the former, then clearly $e\in E(\mathcal H\dotmns x)$. If the latter, then $|f|\ge 3$ since $x,u,v\in f$, and $f\setminus\{x\}\in E(\mathcal H\dotmns x)$. In either case, we have $uv\in E([\mathcal H\dotmns x]_2)$.
\end{proof}

\begin{proof}[Proof of Proposition \ref{connprp}]
By Lemma \ref{lm1}, we see that $x$ is a corner of $[\mathcal H]_2$, say with cover $y\in V([\mathcal H]_2)\setminus\{x\}$. By Lemma \ref{lm2}, we have $[\mathcal H\dotmns x]_2=[\mathcal H]_2-x$. Thus it suffices to show that $[\mathcal H]_2-x$ is connected. Suppose on the contrary that there exist two components $G_1$ and $G_2$ in $[\mathcal H]_2-x$. We may assume that $y\in V(G_1)$. Since $\mathcal H$ is connected implies that $[\mathcal H]_2$ is also connected, there exists $z\in V(G_2)$ such that $xz\in E([\mathcal H]_2)$. Thus, $z\in N_{[\mathcal H]_2}[x]\subset N_{[\mathcal H]_2}[y]$, which implies $yz\in E([\mathcal H]_2)$, a contradiction.
\end{proof}

We are now ready to prove the result about the charaterisation of cop-win hypergraphs. We also include two other equivalent conditions involving the $2$-section. For a hypergraph $\mathcal H$ and an ordering $v_1,\dots,v_n$ of the vertices of $\mathcal H$, define the hypergraphs $\mathcal H_i$ inductively by $\mathcal H_1=\mathcal H$, and $\mathcal H_{i+1}=\mathcal H_i\dotmns v_i$ for $1\le i<n$. We say that $\mathcal H$ is \emph{dismantlable} if we can choose the ordering $v_1,\dots,v_n$ such that, for every $1\le i<n$, $v_i$ is a corner of $\mathcal H_i$.

\begin{theorem}\label{charthm}
Let $\mathcal H$ be a connected hypergraph. The following are equivalent.
\begin{enumerate}
\item[(i)] $\mathcal H$ is a cop-win hypergraph.
\item[(ii)] $\mathcal H$ is dismantlable.
\item[(iii)] $[\mathcal H]_2$ is a cop-win graph.
\item[(iv)] $[\mathcal H]_2$ is dismantlable.
\end{enumerate}
\end{theorem}

\begin{proof}
Note that (i) $\Longleftrightarrow$ (iii) follows from Observation \ref{csameobs}, and (iii) $\Longleftrightarrow$ (iv) follows from Theorem \ref{NWthm}. It suffices to prove that (ii) $\Longleftrightarrow$ (iv). Let $v_1,\dots,v_n$ be the same ordering of the vertices of both $\mathcal H$ and $[\mathcal H]_2$. For $1\le i<n$, let the hypergraphs $\mathcal H_i$ be defined as above, and $G_i=[\mathcal H]_2-\{v_1,\dots,v_{i-1}\}$ (with $G_1=[\mathcal H]_2$). 

We first show that $[\mathcal H_i]_2=G_i$ for $1\le i<n$. This holds for $i=1$. If $[\mathcal H_i]_2=G_i$ for some $1\le i\le n-2$, then using Lemma \ref{lm2}, $[\mathcal H_{i+1}]_2=[\mathcal H_i\dotmns v_i]_2=[\mathcal H_i]_2-v_i=G_i-v_i=G_{i+1}$, and the assertion holds by induction.

Now, we have $\mathcal H$ is dismantlable with sequence $v_1,\dots,v_n$ $\Longleftrightarrow$ $v_i$ is a corner of $\mathcal H_i$ for all $1\le i<n$ $\Longleftrightarrow$ $v_i$ is a corner of $[\mathcal H_i]_2=G_i$ for all $1\le i<n$ (by Lemma \ref{lm1}) $\Longleftrightarrow$ $[\mathcal H]_2$ is dismantlable with sequence $v_1,\dots,v_n$.
\end{proof}

Nowakowski and Winkler \cite{NW83} (see also Aigner and Fromme \cite{AF84}) observed some other facts, stated in the next proposition.

\begin{prop}\label{NWprp}\textup{\cite{AF84,NW83}}
\indent
\begin{enumerate}
\item[(a)] Let $G$ be a cop-win graph. Then $G$ has a corner.
\item[(b)] Let $x$ be a corner of a connected graph $G$. Then $G$ is cop-win if and only if $G-x$ is cop-win.
\end{enumerate}
\end{prop}

Proposition \ref{NWprp} can be extended to hypergraphs.

\begin{prop}\label{NWhgprp}
\indent
\begin{enumerate}
\item[(a)] Let $\mathcal H$ be a cop-win hypergraph. Then $\mathcal H$ has a corner.
\item[(b)] Let $x$ be a corner of a connected hypergraph $\mathcal H$. Then $\mathcal H$ is cop-win if and only if $\mathcal H \dotmns x$ is cop-win.
\end{enumerate}
\end{prop}

\begin{proof}
(a) As in the graphs case, we may consider the final round in a cop winning strategy, as mentioned in \cite{AF84,NW83}. The robber moves from some vertex $v\in V(\mathcal H)$ to a vertex in $N[v]$ (possibly remaining at $v$ itself). The cop is at some vertex $u\in V(\mathcal H)\setminus\{v\}$ and is able to capture the robber with her move. This means that $N[v]\subset N[u]$, and $v$ is a corner of $\mathcal H$ with cover $u$.\\[1ex]
\indent (b) By Lemma \ref{lm1}, we have $x$ is a corner of $\mathcal H$ $\Longleftrightarrow$ $x$ is a corner of $[\mathcal H]_2$. Therefore, $\mathcal H$ is cop-win $\Longleftrightarrow$ $[\mathcal H]_2$ is cop-win (by Theorem \ref{charthm}) $\Longleftrightarrow$ $[\mathcal H]_2-x$ is cop-win (by Proposition \ref{NWprp}(b)) $\Longleftrightarrow$ $[\mathcal H\dotmns x]_2$ is cop-win (by Lemma \ref{lm2}) $\Longleftrightarrow$ $\mathcal H\dotmns x$ is cop-win (by Theorem \ref{charthm}).
\end{proof}

To conclude this section, we remark that instead of the deletion operator $\dotmns$\,, we may consider the \emph{weak deletion operator} $\wmns$. For a hypergraph $\mathcal H$ and $x\in V(\mathcal H)$, the \emph{weak delation} of $x$ from $\mathcal H$ yields the hypergraph $\mathcal H\wmns x$ where
\begin{align*}
V(\mathcal H\wmns x) &= V(\mathcal H) \setminus\{x\},\\
E(\mathcal H\wmns x) &= \{e\in E(\mathcal H):x\not\in e\}\dcup\{f\setminus\{x\}: f\in E(\mathcal H),\,x\in f\}.
\end{align*}

We may then use $\wmns$ in place of $\dotmns$ in the definition of a dismantlable hypergraph, and all results in this section would hold with $\wmns$. One minor issue with using $\wmns$ is that when $\mathcal H=G$ is a graph, we do not have $G\wmns x=G-x$, since $G\wmns x$ will have edges of size $1$ at every vertex of $N_G(x)$, and so $G\wmns x$ is not a graph if $N_G(x)\neq\emptyset$. Essentially, edges of size $1$ in the game of Cops and Robber on hypergraphs are irrelevant.

\section{Specific classes of hypergraphs}\label{specificsect}

We shall determine the cop number of some specific classes of hypergraphs in this section, namely, hypertrees and complete multipartite hypergraphs.

Recall that a \emph{leaf} of a tree is a vertex with degree $1$. Every non-trivial tree has at least two leaves, and deleting a leaf from a tree gives another tree. For a tree $T$, we have $x\in V(T)$ is a cut-vertex of $T$ (that is, $T-x$ is a forest with at least two components) if and only if $x$ is not a leaf. For $x,y\in V(T)$, there exists a unique path in $T$ which connects $x$ and $y$. If $P$ is a path and $x,y\in V(P)$ where $x=y$ is possible, we write $xPy$ for the subpath of $P$ with end-vertices $x$ and $y$.

A hypergraph $\mathcal T$ is a \emph{hypertree} if $\mathcal T$ is connected, and there exists a tree $T$ such that $V(T)=V(\mathcal T)$, with every edge of $\mathcal T$ inducing a subtree of $T$. Such a tree $T$ is called a \emph{host tree} of $\mathcal T$.

\begin{obs}\label{hypertreeobs}
Let $\mathcal T$ be a hypertree with host tree $T$.
\begin{enumerate}
\item[(a)] If two vertices $u,v\in V(\mathcal T)$ belong to an edge $e\in E(\mathcal T)$, then all vertices of the unique path connecting $u$ and $v$ in $T$ also belong to $e$.
\item[(b)] For every edge $uv\in E(T)$, there exists $e\in E(\mathcal T)$ such that $u,v\in e$. 
\end{enumerate}
\end{obs}

We are now ready to prove the following theorem, which asserts that all hypertrees are cop-win hypergraphs.

\begin{theorem}\label{hypertreethm}
Let $\mathcal T$ be a hypertree. Then $c(\mathcal T)=1$.
\end{theorem}

Since Theorem \ref{hypertreethm} appears to be a fundamental and interesting result, we shall provide two different proofs. In the first proof, we utilise the result about the characterisation of cop-win hypergraphs that was proved in Theorem \ref{charthm}. In the second proof, we provide a self-contained proof which describes a winning strategy for one cop when she plays against the robber in a hypertree.

\begin{proof}[First proof]
The theorem holds for $|V(\mathcal T)|=1$, so we assume that $|V(\mathcal T)|\ge 2$. We show that $\mathcal T$ is dismantlable, so that the result follows from Theorem \ref{charthm}. Fix a host tree $T$ of $\mathcal T$. Let $x$ be a leaf of $T$, and $y$ be the unique neighbour of $x$ in $T$.

We first show that $x$ is a corner of $\mathcal T$, with cover $y$. Every edge $e\in E(\mathcal T)$ with $|e|\ge 2$ and containing $x$ must also contain $y$. Indeed, such an edge $e$ contains some $z\in V(\mathcal T)\setminus\{x\}$, and the unique path connecting $x$ and $z$ in $T$ contains $y$, so that $y\in e$ by Observation \ref{hypertreeobs}(a). Thus $N_{\mathcal T}[x]\subset N_{\mathcal T}[y]$.

Next, we show that $\mathcal T\dotmns x$ is another hypertree, with host tree $T-x$. Since $x$ is a corner of $\mathcal T$, by Proposition \ref{connprp}, $\mathcal T\dotmns x$ is connected. Clearly, $V(T-x)=V(\mathcal T\dotmns x)$. If $e\in E(\mathcal T\dotmns x)$ such that $e\in E(\mathcal T)$ and $x\not\in e$, then clearly $e$ induces a subtree of $T-x$. Now suppose that $f\setminus\{x\}\in E(\mathcal T\dotmns x)$, where $f\in E(\mathcal T)$, $x\in f$ and $|f|\ge 3$. Then $f$ induces a subtree $T'$ of $T$, and $x\in V(T')$. Since $x$ is a leaf of $T$, we see that $x$ is also a leaf of $T'$. Thus, $f\setminus\{x\}$ induces the subtree $T'-x$ of $T-x$.

Now, starting with $\mathcal T$, we may use the operation $\dotmns$ to successively delete corners, obtaining another hypertree upon each deletion. Therefore, $\mathcal T$ is dismantlable, and we are done.
\end{proof}

\begin{proof}[Second proof]
The theorem holds for $|V(\mathcal T)|=1,2$, so we may assume that $|V(\mathcal T)|\ge 3$. We show that one cop $C$ has a winning strategy against the robber $R$. Fix a host tree $T$ of $\mathcal T$. In the first round, $C$ chooses a cut-vertex of $u_1$ of $T$ (which exists since $|V(T)|\ge 3$), so that $T-u_1$ is a forest with at least two components, each of which is a subtree of $T$. $R$ then chooses a vertex $v_1$ of $T-u_1$, which belongs to a component $T_1$ of $T-u_1$.

Suppose that after the $i$th round, $C$ and $R$ are at distinct vertices $u_i$ and $v_i$, where $u_i$ is a cut-vertex of $T$, and $v_i$ is in subtree $T_i$ of $T$ which is a component of $T-u_i$. We show that $C$ can make a move to a vertex $u_{i+1}$ so that after the $(i+1)$th round, $u_{i+1}$ is a cut-vertex of $T$, and $R$ must be at some vertex $v_{i+1}$ in a component $T_{i+1}$ of $T-u_{i+1}$, with $V(T_{i+1})\varsubsetneq V(T_i)$, if he wants to avoid capture. 

Consider the unique path $P$ connecting $u_i$ and $v_i$ in $T$. Let $w$ be the neighbour of $u_i$ in $P$. Then by Observation \ref{hypertreeobs}(b), $u_i$ and $w$ belong to some edge of $\mathcal T$. Also, if both $u_i$ and some vertex $x\in V(P)$ belong to some edge $e\in E(\mathcal T)$, then by Observation \ref{hypertreeobs}(a), all vertices of $u_iPx$ also belong to $e$. It follows that there exists a vertex $u_{i+1}\in V(wPv_i)$ and an edge $f\in E(\mathcal T)$ such that, all vertices of $u_iPu_{i+1}$ belong to $f$, and no vertex of $P-V(u_iPu_{i+1})$ belongs to $N_{\mathcal T}[u_i]$. The cop $C$ makes the move from $u_i$ to $u_{i+1}$ along $f$. If $u_{i+1}=v_i$, then $C$ has captured the robber $R$. Otherwise, $u_{i+1}$ has two neighbours $y,z\in V(P)$, with $y\in V(u_iPu_{i+1})$ and $z\in V(u_{i+1}Pv_i)$. We have $u_{i+1}$ is a cut-vertex of $T$. Now, note that since $u_{i+1}\in V(wPv_i) \subset V(T_i)$, all vertices of $T-V(T_i)$, and also $y$, are in the same component of $T-u_{i+1}$, say $S$. Meanwhile, the vertices $v_i$ and $z$ are in a component of $T-u_{i+1}$, say $T_{i+1}$, which is different from $S$ since $y$ and $z$ are in different components. It follows that $V(T_{i+1})\subset V(T_i)\setminus\{u_{i+1}\}\varsubsetneq V(T_i)$. Suppose now $R$ makes a move to the vertex $v_{i+1}$ along the edge $g\in E(\mathcal T)$. If $v_{i+1}$ is in a component of $T-u_{i+1}$ which is different from $T_{i+1}$, then the unique path connecting $v_i$ and $v_{i+1}$ in $T$ must contain $u_{i+1}$. Thus by Observation \ref{hypertreeobs}(a), we have $u_{i+1},v_i,v_{i+1}\in g$, and $C$ can capture $R$ in her next move by using $g$. Therefore, $v_{i+1}\in V(T_{i+1})$, and the robber's move is restricted to within $T_{i+1}$.

These rounds must end at some point with $C$ capturing $R$ as $i$ increases, since the set of vertices available for the robber to move to, in order to avoid capture, has strictly decreased after each round.
\end{proof}

We\, now\, consider\, the\, game\, of\, Cops\, and\, Robber\, played\, on\, complete\, multipartite hypergraphs. For $t\ge r\ge 2$ and $1\le n_1\le\cdots\le n_t$, the \emph{$r$-uniform complete $t$-partite hypergraph}, denoted by $\mathcal K^r_{n_1,\dots,n_t}$, has vertex set $\bigcup_{i=1}^t V_i$, where the sets $V_i$ are disjoint and $|V_i|=n_i$ for every $1\le i\le t$. Each set $V_i$ is a \emph{class} of $\mathcal K^r_{n_1,\dots,n_t}$. The edge set of $\mathcal K^r_{n_1,\dots,n_t}$ consists of all edges of size $r$ that meet every class  $V_i$ in at most one vertex. We may also call $\mathcal K^r_{n_1,\dots,n_t}$ an \emph{$r$-uniform complete multipartite hypergraph}.

If $2\le s\le t<r$, then we may also define a related hypergraph. Let $1\le n_1\le\cdots\le n_t$ be such that $\sum_{i=1}^t n_i\ge r$. The $r$-uniform hypergraph $\mathcal L^{r,s}_{n_1,\dots,n_t}$ is defined as follows. Again, the vertex set of $\mathcal L^{r,s}_{n_1,\dots,n_t}$ is $\bigcup_{i=1}^t V_i$, where the sets $V_i$ are disjoint and $|V_i|=n_i$ for every $1\le i\le t$. Each set $V_i$ is a \emph{class} of $\mathcal L^{r,s}_{n_1,\dots,n_t}$. The edge set of $\mathcal L^{r,s}_{n_1,\dots,n_t}$ consists of all edges of size $r$ that meet at least $s$ of the classes $V_i$. 

We have the following result about the cop numbers of these hypergraphs.

\begin{prop}\label{mpprp}
Let $t\ge 2$ and $1\le n_1\le\cdots\le n_t$.
\begin{enumerate}
\item[(a)] Let $t\ge r\ge 2$. Then $c(\mathcal K^r_{n_1,\dots,n_t})=1$ if $n_1=1$, and $c(\mathcal K^r_{n_1,\dots,n_t})=2$ if $n_1\ge 2$.
\item[(b)] Let $2\le s\le t<r$ and $\sum_{i=1}^t n_i\ge r$. Then $c(\mathcal L^{r,s}_{n_1,\dots,n_t})=1$. 
\end{enumerate}
\end{prop}

\begin{proof}
\, Let\, the\, classes\, $V_1,\dots,V_t$\, be\, defined\, as\, above,\, and\, write\, $\mathcal K$\, and\, $\mathcal L$\, for\, the\, two hypergraphs for simplicity.\\[1ex]
\indent (a) If $n_1=1$, then one cop can win the game by choosing the vertex $u\in V_1$ in the first round, so that $N[u]=V(\mathcal K)$. Thus $c(\mathcal K)=1$. Now let $n_1\ge 2$. If there is only one cop, then the robber can always avoid capture by making sure that at the end of every round, he is in the same class of $\mathcal K$ as the cop, with the two occupied vertices distinct. If there are two cops, then they can win the game by choosing vertices $u\in V_1$ and $v\in V_2$ in the first round, so that $N[u]\cup N[v]=V(\mathcal K)$. Thus $c(\mathcal K)=2$.\\[1ex]
\indent (b) Note that the given conditions imply that $n_t\ge 2$. Also, for any vertices $v_i\in V_i$ for $1\le i<t$ and $v_t,v_t'\in V_t$, there exists an edge that contains all of these vertices, since $t+1\le r$. Thus one cop can win the game by simply choosing a vertex $u\in V_t$ in the first round, so that $N[u]=V(\mathcal L)$. Thus $c(\mathcal L)=1$.
\end{proof}

\section{Cartesian products of hypergraphs}\label{prodsect}

In this section, we shall prove some results concerning the cop number of the Cartesian product of hypergraphs.

For graphs $G$ and $H$, the \emph{Cartesian product} $G\,\square\,H$ is a graph with vertex set $V(G)\times V(H)$, and vertices $(a, b)$ and $(a', b')$ are adjacent if $a = a'$ and $bb'\in E(H)$, or $aa'\in E(G)$ and $b = b'$. This definition extends easily to hypergraphs.  For hypergraphs $\mathcal G$ and $\mathcal H$, the \emph{Cartesian product} $\mathcal G\,\square\,\mathcal H$ is a hypergraph with vertex set $V(\mathcal G)\times V(\mathcal H)$. The edge set of $\mathcal G\,\square\,\mathcal H$ consists of all edges of the form $\{a\}\times f$ for some $a\in V(\mathcal G)$ and $f\in E(\mathcal H)$, or the form $e\times \{b\}$ for some $e\in E(\mathcal G)$ and $b\in V(\mathcal H)$. 


We may extend the Cartesian product to more than two hypergraphs. The operation $\,\square\,$ is both associative and commutative. Therefore, for hypergraphs $\mathcal H_1,\dots, \mathcal H_d$, the notation $\mathcal H_1\,\square\,\cdots\,\square\,\mathcal H_d$ makes sense and is invariant under any permutation of $\mathcal H_1,\dots, \mathcal H_d$. If the hypergraphs\, $\mathcal H_1,\dots, \mathcal H_d$\, are\, connected,\, then\, $\mathcal H_1\,\square\,\cdots\,\square\,\mathcal H_d$\, is\, also\, connected.\, We\, may visualise the structure of $\mathcal H_1\,\square\,\cdots\,\square\,\mathcal H_d$ as follows. We have a $|V(\mathcal H_1)|\times\cdots\times |V(\mathcal H_d)|$ grid of vertices, and we think of the grid as a $d$-dimensional structure with $d$ directions. For every $1\le i\le d$, we add a copy of $\mathcal H_i$ to every set of vertices in the $i$th direction.

There are many results about the cop number of the Cartesian product of graphs. For the cases of a product of trees and a product of complete graphs, the exact cop numbers have been determined by Maamoun and Meyniel \cite{MM87}, and Neufeld and Nowakowski \cite{NN98}. We state their results in the following theorem, as well as a result of To\v{s}i\'c \cite{Tos88}.

\begin{theorem}\label{prodgraphthm}
\indent
\begin{enumerate}
\item[(a)] \textup{\cite{Tos88}} Let $G$ and $H$ be connected graphs. Then 
\[
\max(c(G),c(H))\le c(G\,\square\,H) \le c(G)+c(H).
\] 
\item[(b)] \textup{\cite{MM87}} Let $T_1,\dots, T_d$ be non-trivial trees. Then 
\[
c(T_1\,\square\cdots\square\,T_d)=\Big\lceil\frac{d+1}{2}\Big\rceil.
\]
\item[(c)] \textup{\cite{NN98}} Let $G_1,\dots,G_d$ be complete graphs on at least $3$ vertices. Then 
\[
c(G_1\,\square\cdots\square\,G_d)=d.
\]
\end{enumerate}
\end{theorem}

Here, we consider the cop number of the Cartesian product of hypergraphs. Siriwong, Boonklurb and Singhun \cite{SBS20} proved some results in this direction. Namely, for the Cartesian product, and versions of the direct and  strong products of hypergraphs, they showed whether or not two cop-win hypergraphs yields another cop-win hypergraph under each of these products. In Theorem \ref{prodhgthm} below, we shall extended the results of Theorem \ref{prodgraphthm} to hypergraphs.

\begin{theorem}\label{prodhgthm}
\indent
\begin{enumerate}
\item[(a)] Let $\mathcal G$ and $\mathcal H$ be connected hypergraphs on at least $2$ vertices. Then 
\[
\max(c(\mathcal G),c(\mathcal H),2)\le c(\mathcal G\,\square\,\mathcal H) \le c(\mathcal G)+c(\mathcal H).
\]
In particular, if $c(\mathcal G)=c(\mathcal H)=1$, then $c(\mathcal G\,\square\,\mathcal H)=2$.
\item[(b)] Let $p\ge 0$ and $q\ge 1$. Let $T_1,\dots,T_p$ be non-trivial trees, and $\mathcal H_1,\dots, \mathcal H_q$ be connected hypergraphs with $c(\mathcal H_j)=1$ and $\delta([\mathcal H_j]_2)\ge 2$ for every $1\le j\le q$. Let $\mathcal H$ be the Cartesian product of these trees and hypergraphs. Then 
\[
c(\mathcal H)=\Big\lceil\frac{p}{2}\Big\rceil+q.
\]
\end{enumerate}
\end{theorem}

Theorem\, \ref{prodhgthm}(b)\, has\, the\, following\, corollary,\, the\, two\, parts\, of\, which\, can\, be\, seen\, as generalisations of Theorem \ref{prodgraphthm}(b) and (c).

\begin{corollary}\label{prodhgcor}
\indent
\begin{enumerate}
\item[(a)] Let $\mathcal T_1,\dots, \mathcal T_d$ be hypertrees with anti-rank at least $3$. Then 
\[
c(\mathcal T_1\,\square\cdots\square\,\mathcal T_d)=d.
\]
\item[(b)] Let $p\ge 0$ and $q\ge 1$. Let $G$ be the Cartesian product of $p$ copies of $K_2$ (that is, $G$ is the discrete cube $Q_p$), and $\mathcal H_1,\dots, \mathcal H_q$ be uniform complete hypergraphs with $r(\mathcal H_j)\ge 2$ and $|V(\mathcal H_j)|\ge 3$ for all $1\le j\le q$ (two different $\mathcal H_j$ may have different ranks). Then 
\[
c(G\,\square\,\mathcal H_1\,\square\cdots\square\,\mathcal H_q)=\Big\lceil\frac{p}{2}\Big\rceil+q.
\]
\end{enumerate}
\end{corollary}

Note that Corollary \ref{prodhgcor}(a) follows from Theorems \ref{hypertreethm} and \ref{prodhgthm}(b). Also, Theorem \ref{prodgraphthm}(c) omitted the case that some of the complete graphs may have $2$ vertices. This missing case is covered in Corollary \ref{prodhgcor}(b). Before we prove Theorem \ref{prodhgthm}, we need two results. Proposition \ref{prodhgprp} below is fairly trivial, while Lemma \ref{MMlem} is a result of Maamoun and Meyniel \cite{MM87}.

\begin{prop}\label{prodhgprp}
Let $\mathcal G$ and $\mathcal H$ be hypergraphs. Then $[\mathcal G\,\square\,\mathcal H]_2=[\mathcal G]_2\,\square\,[\mathcal H]_2$.
\end{prop}

\begin{lemma}\label{MMlem}\textup{\cite{MM87}}
Let $T_1$ and $T_2$ be trees. Then in $T_1\,\square\,T_2$, one cop may capture the robber if the robber is not permitted to remain at a vertex indefinitely.
\end{lemma}

In Lemma \ref{MMlem}, the winning strategy for the cop, as stated in \cite{MM87}, is as follows. Suppose that at the end of a round, the cop is at a distance $d_1+d_2$ from the robber in $T_1\,\square\,T_2$, where $d_i$ is the distance between the two vertices in $T_i$ corresponding to the cop and the robber, for $i=1,2$. The cop moves as follows. If $d_1 + d_2$ is even, then she passes her turn. If $d_1 + d_2$ is odd, then she moves to a vertex in $T_1\,\square\,T_2$ which decreases $\max(d_1, d_2)$. This strategy implies that the distance between the cop and the robber in $T_1\,\square\,T_2$ does not increase, and from time to time will strictly decrease until it reaches zero.

\begin{proof}[Proof of Theorem \ref{prodhgthm}]
(a) Using Observation \ref{csameobs}, Theorem \ref{prodgraphthm}(a) and Proposition \ref{prodhgprp}, it is easy to obtain $\max(c(\mathcal G),c(\mathcal H))\le c(\mathcal G\,\square\,\mathcal H) \le c(\mathcal G)+c(\mathcal H)$. It remains to prove that $c(\mathcal G\,\square\,\mathcal H)\ge 2$. Suppose that one cop $C$ plays the game against the robber $R$ on $\mathcal G\,\square\,\mathcal H$. Since $|V(\mathcal G)|,|V(\mathcal H)|\ge 2$, the robber $R$ can always avoid capture by ensuring that at the end of each round, if $C$ is at some $(a,b)$, then $R$ is at some $(a',b')$ where $a\neq a'$ and $b\neq b'$.\\[1ex]
\indent (b) Our proof of this result follows ideas of Maamoun and Meyniel \cite{MM87}. The case $(p,q)=(0,1)$ holds, so we may assume that $(p,q)\neq(0,1)$.


We first prove that $\lceil\frac{p}{2}\rceil+q-1$ cops are insufficient to capture the robber $R$ in $\mathcal H$. It is not difficult to verify the following inequality, say by induction on $p+q$:
\begin{equation}\label{prodhgthmlb}
\Big(\Big\lceil\frac{p}{2}\Big\rceil+q-1\Big)\Big(1+\sum_{i=1}^p(|V(T_i)|-1)+\sum_{j=1}^q(|V(\mathcal H_j)|-1)\Big) <\prod_{i=1}^p |V(T_i)|\prod_{j=1}^q |V(\mathcal H_j)|.
\end{equation}
The\, left\, hand\, side\, of\, (\ref{prodhgthmlb})\, is\, an\, upper\, bound\, on\, the\, size\, of\, the\, union\, of\, the\, closed neighbourhoods of $\lceil\frac{p}{2}\rceil+q-1$ cops, when they are placed in $\mathcal H$. Therefore in the first round, after the $\lceil\frac{p}{2}\rceil+q-1$ cops have chosen their vertices in $\mathcal H$, the robber $R$ can choose a vertex to avoid capture in the second round.

Subsequently in each round, the robber $R$ avoids capture as follows. On his move, if no cop is at a vertex adjacent to him, then he passes his move. Otherwise, suppose there is a cop $C_1$ adjacent to him, say using the edge $e$. The robber $R$ will attempt to escape by moving along some other edge $f$ which is in a different direction to $e$. Let us call any vertex that $R$ may reach by using such an edge $f$ an \emph{escape vertex}. Now, consider the minimum number of cops that are necessary so that all escape vertices are \emph{covered}: they are either occupied by or are adjacent to the cops. For a direction corresponding to some $T_i$, there may be as few as only one escape vertex, namely, if $R$ is on a vertex which is a leaf of $T_i$. Thus for two directions corresponding to two of the $T_i$, one cop may be sufficient to cover all escape vertices in either direction. It can be the case that there are only two escape vertices $x$ and $y$, one in each direction, and the cop is at the vertex in the plane spanned by the two directions and adjacent to $x$ and $y$. For a direction corresponding to some $\mathcal H_j$, the condition $\delta([\mathcal H_j]_2)\ge 2$ implies that there are at least two escape vertices in the direction. This means that to cover these escape vertices, there should be a cop in the copy of $\mathcal H_j$ containing $R$. Therefore, if $e$ belongs to some $T_i$ (whence $p\ge 1$), then at least $\lceil\frac{p-1}{2}\rceil+q$ cops are necessary to cover all escape vertices. If $e$ belongs to some $\mathcal H_j$, then at least $\lceil\frac{p}{2}\rceil+q-1$ cops are necessary. Since we only have $\lceil\frac{p}{2}\rceil+q-2$ remaining cops (not including $C_1$) to attempt to cover all escape vertices, this number is insufficient in either case. Therefore, the robber $R$ can always escape to a vertex to avoid capture in the next round. The lower bound $c(\mathcal H)\ge\lceil\frac{p}{2}\rceil+q$ follows.

Now we prove the upper bound 
\begin{equation}\label{prodhgthmub}
c(\mathcal H)\le\Big\lceil\frac{p}{2}\Big\rceil+q.
\end{equation}
Note that part (a) and Theorem \ref{prodgraphthm}(c) imply that $c(\mathcal H)\le\lceil\frac{p+1}{2}\rceil+q$, so (\ref{prodhgthmub}) holds for $p$ odd. We prove (\ref{prodhgthmub}) for even $p$ by induction. Again by (a), we see that (\ref{prodhgthmub}) holds for $p=0$. Let $p\ge 2$ be even and suppose that (\ref{prodhgthmub}) holds for $p-2$. We show that $\frac{p}{2}+q$ cops have a winning strategy on $\mathcal H$. Write $\mathcal H=T_1\,\square\,T_2\,\square\,\mathcal G$, where $\mathcal G$ is the Cartesian product of the remaining trees and hypergraphs. We say that a cop or the robber \emph{moves in direction $1,2$ or $3$} if they move along an edge in a copy of $T_1,T_2$ or $\mathcal G$, respectively. By induction, we have $c(\mathcal G)\le \frac{p}{2}+q-1$. 
 
In the first round, the $\frac{p}{2}+q$ cops choose vertices in the same copy of $\mathcal G$ in $\mathcal H$, say $\tilde{\mathcal G}$, and then the robber $R$ chooses a vertex of $\mathcal H$. Let $\tilde{R}$ be the position in $\tilde{\mathcal G}$ which corresponds to the robber $R$. The $\frac{p}{2}+q$ cops play against $\tilde{R}$ within $\tilde{\mathcal G}$, and by induction, one cop, say $C_1$, is able to capture $\tilde{R}$. We are done if $\tilde{R}$ is the robber $R$ himself. Otherwise, proceed as follows.

From this point onwards, whenever $R$ moves in direction $3$, $C_1$ will always also move in direction $3$ so that she is in the same copy of $T_1\,\square\,T_2$ as $R$. Suppose that at some stage, $R$ moves only in direction $3$ or does not move for too many consecutive rounds. Then the remaining $\frac{p}{2}+q-1$ cops (other than $C_1$) can capture $R$, by firstly moving together in directions $1$ or $2$ to the same copy of $\mathcal G$ as $R$, then capture $R$ within the copy of $\mathcal G$ (by induction). Therefore, from time to time, $R$ must move in direction $1$ or $2$. But then, $C_1$ can capture $R$ according to the winning strategy in Lemma \ref{MMlem}, by moving in direction $1$ or $2$ whenever $R$ does so, or whenever $R$ does not move. Every pair of consecutive moves in direction $3$ made by $R$, then by $C_1$, have cancelled each other out, so the game played between $C_1$ and $R$ is equivalent to the one played on  $T_1\,\square\,T_2$ in Lemma \ref{MMlem}.
\end{proof}

We conclude this section by considering a Cartesian product-like hypergraph called the \emph{prism over a hypergraph}. This prism hypergraph was considered by Boonklurb, Singhun and Termtanasombat \cite{BST15}, when they considered the problem of finding a  decomposition of the hypergraph into Hamiltonian cycles. Here, we shall give a slightly more general definition of this prism hypergraph, and then consider the game of Cops and Robber on it.

Let $\mathcal H$ be a hypergraph, and $n,r\ge 2$. Let $\mathcal H_1,\dots, \mathcal H_n$ be disjoint copies of $\mathcal H$. For $i\neq j$, we say that vertices $v^i\in V(\mathcal H_i)$ and $v^j\in V(\mathcal H_j)$ (resp.~edges $e^i\in E(\mathcal H_i)$ and $e^j\in E(\mathcal H_j)$) are \emph{clones} if $v^i$ and $v^j$ (resp.~$e^i$ and $e^j$) represent the same vertex (resp.~edge) of $\mathcal H$. The \emph{prism hypergraph} $\mathcal P(\mathcal H,n,r)$ is the hypergraph constructed as follows. We take $\mathcal H_1,\dots, \mathcal H_n$. Then for every $1\le i<n$, we add edges of size $r$ of the form $v^iv^{i+1}u_1\cdots u_{r-2}$, where $v^i\in V(\mathcal H_i)$ and $v^{i+1}\in V(\mathcal H_{i+1})$ are clone vertices, and $u_1,\dots,u_{r-2}\in (e^i\cup e^{i+1})\setminus\{v^i,v^{i+1}\}$, where $e^i\in E(\mathcal H_i)$ and $e^{i+1}\in E(\mathcal H_{i+1})$ are clone edges such that $v^i\in e^i$, $v^{i+1}\in e^{i+1}$ and $|e^i\cup e^{i+1}|\ge r$. Such an edge $v^iv^{i+1}u_1\cdots u_{r-2}$ between $\mathcal H_i$ and $\mathcal H_{i+1}$ is an \emph{$i$-transitional edge (pivoted at $v^i$)}. In the case $r=2$, we see that $\mathcal P(\mathcal H,n,2)$ looks like a prism since it is the Cartesian product $\mathcal H\,\square\,P_n$, where $P_n$ is the path on $n$ vertices. 

If $\mathcal H$ is connected and has rank $r(\mathcal H)\ge \frac{r}{2}$, then clearly there is an $i$-transitional edge for every $1\le i<n$, so that $\mathcal P(\mathcal H,n,r)$ is also connected. We have the following result for the cop number of $\mathcal P(\mathcal H,n,r)$.

\begin{theorem}\label{prismthm}
Let $n\ge 2$ and $r\ge 3$. Let $\mathcal H$ be a connected hypergraph with anti-rank $s(\mathcal H)\ge\frac{r}{2}$. Then $c(\mathcal P(\mathcal H,n,r))=c(\mathcal H)$.
\end{theorem}

\begin{proof}
We first show that $c(\mathcal H)$ cops are sufficient to capture the robber $R$ in $\mathcal P(\mathcal H,n,r)$. The cops have a winning strategy as follows. First, the $c(\mathcal H)$ cops choose vertices of $\mathcal H_1$, then the robber chooses a vertex of $\mathcal P(\mathcal H,n,r)$. The cops then play within $\mathcal H_1$ and pursue the clone of the robber's position. At some stage, one of the cops, say $C_1$, successfully captures the robber's clone in $\mathcal H_1$. We are done if the robber $R$ himself is already in $\mathcal H_1$. Otherwise from here on, the strategy is to advance $C_1$ from $\mathcal H_1$ to $\mathcal H_2$, then to $\mathcal H_3$ and so on, while tracking the clone of the robber's position. Suppose that at some stage, $C_1$ is at $v^i\in V(\mathcal H_i)$ and $R$ is at $v^j\in V(\mathcal H_j)$ which is a clone of $v^i$, for some $1\le i<j\le n$. It is the robber's turn to move, and he may move to $\mathcal H_{j-1},\mathcal H_j$ or $\mathcal H_{j+1}$. Since $r\ge 3$, the cop $C_1$ may respond to the robber's move as follows.
\begin{itemize}
\item Suppose that $R$ moves to $w^{j-1}\in V(\mathcal H_{j-1})$ along a $(j-1)$-transitional edge, pivoted at some $x^{j-1}\in V(\mathcal H^{j-1})$. Any two of $v^{j-1}, w^{j-1},x^{j-1}$ may be the same. Then there is an edge $e^{j-1}\in E(\mathcal H_{j-1})$ containing $v^{j-1},w^{j-1},x^{j-1}$. If $j=i+1$, then $C_1$ captures $R$ by moving along $e^i$. Otherwise, $j>i+1$, and $C_1$ moves to $w^{i+1}$ along an $i$-transitional edge pivoted at $v_i$ and containing $v^i,v^{i+1},w^{i+1}$, which exists since $v^i,w^i\in e^i$ and $|e^i\cup e^{i+1}|\ge r$. If $j=i+2$, then $C_1$ again captures $R$. Otherwise, $j>i+2$.
\item Suppose that $R$ moves to $w^j\in V(\mathcal H_j)$ along an edge $e^j\in E(\mathcal H_j)$, where $w^j=v^j$ is possible. Then $C_1$ moves to $w^{i+1}\in V(\mathcal H_{i+1})$ along an $i$-transitional edge pivoted at $v^i$ and containing $v^i, v^{i+1},w^{i+1}$, which exists since since $v^i,w^i\in e^i$ and $|e^i\cup e^{i+1}|\ge r$. If $j=i+1$, then $C_1$ captures $R$. Otherwise, $j>i+1$.
\item Suppose that $j<n$, and $R$ moves to $w^{j+1}\in V(\mathcal H_{j+1})$ along a $j$-transitional edge, pivoted at some $x^j\in V(\mathcal H^j)$. Any two of $v^{j+1}, w^{j+1},x^{j+1}$ may be the same. Then there is an edge $e^{j+1}\in E(\mathcal H_{j+1})$ containing $v^{j+1},w^{j+1},x^{j+1}$. $C_1$ moves to $w^{i+1}$ along an $i$-transitional edge pivoted at $v_i$ and containing $v^i,v^{i+1},w^{i+1}$, which exists since $v^i,w^i\in e^i$ and $|e^i\cup e^{i+1}|\ge r$.
\end{itemize}
If the cop $C_1$ has not captured the robber $R$ after a pair of moves as described above, then these moves are iterated. On each iteration where the robber has not been captured, $C_1$ has advanced from some $\mathcal H_i$ to $\mathcal H_{i+1}$ and tracked the clone of $R$, while $R$ is confined to some $\mathcal H_j$ where $i+1<j\le n$. This procedure must terminate as $i$ increases towards $n$, with the cop $C_1$ capturing the robber $R$.

We have now proved Theorem \ref{prismthm} for $c(\mathcal H)=1$. It remains to prove that for $c(\mathcal H)\ge 2$,  $c(\mathcal H)-1$ cops cannot capture the robber in $\mathcal P(\mathcal H,n,r)$. The robber has a winning strategy as follows. First, the $c(\mathcal H)-1$ cops choose vertices of $\mathcal P(\mathcal H,n,r)$. The robber then considers the clones in $\mathcal H_1$ of the cops' positions, and chooses a vertex of $\mathcal H_1$ for which he has a winning strategy if he was playing the game within $\mathcal H_1$ against the cops' clones. From here on, the robber simply plays within $\mathcal H_1$ by evading the cops' clones in $\mathcal H_1$. Indeed, suppose that at some stage, the robber makes a move to $v^1\in V(\mathcal H_1)$. Then certainly, $v^1$ is not adjacent to any vertex of $\mathcal H_1$ occupied by a cop. Suppose that $v^1$ is adjacent to some vertex $w^2\in V(\mathcal H_2)$ which is occupied by a cop. Then there exists a $1$-transitional edge pivoted at some $x^1\in V(\mathcal H_1)$, and containing $v^1,w^2,x^1,x^2$. Any two of $v^1,w^1,x^1$ may be the same, and similarly for $v^2,w^2,x^2$. But this implies that $v^1,w^1,x^1$ belong to some edge of $\mathcal H_1$, which is a contradiction since the robber made the move to $v^1$ to avoid capture by the clone of the cop at $w^1$. Therefore, each time the robber makes a move, he is able to avoid capture by a cop, and this may be repeated indefinitely.

This completes the proof of Theorem \ref{prismthm}.
\end{proof}

For\, the\, case\, $r=2$,\, recall\, that\, $\mathcal P(\mathcal H,n,2)=\mathcal H\,\square\,P_n$.\, Since\, $c(P_n)=1$,\, if\, $\mathcal H$\, is connected, then Theorem \ref{prodhgthm}(a) implies that $c(\mathcal P(\mathcal H,n,2))\in \{c(\mathcal H),c(\mathcal H)+1\}$. Both values $c(\mathcal H)$ and $c(\mathcal H)+1$ can be attained. If $\mathcal H$ is any hypergraph with $c(\mathcal H)=1$ (for example, $\mathcal H$ may be a uniform complete hypergraph, or a hypertree), then Theorem \ref{prodhgthm}(a) implies that $c(\mathcal P(\mathcal H,n,2))=c(\mathcal H\,\square\,P_n)=2=c(\mathcal H)+1$. On the other hand, if $\mathcal H=C_\ell$ is a graph cycle with length $\ell\ge 4$, then a result of Neufeld and Nowakowski about the cop number of the Cartesian product of cycles and trees (see \cite{NN98}, Theorem 2.8) implies that $c(\mathcal P(\mathcal H,n,2))=c(C_\ell\,\square\,P_n)=2=c(\mathcal H)$.

\section{Conclusion and future research}
In this paper, we considered the game of Cops and Robber played on hypergraphs and proved some results. In Theorem \ref{charthm}, we have proved a characterisation result of cop-win hypergraphs and described a structural property called \emph{dismantlable} for these hypergraphs. In Theorem \ref{hypertreethm}, we proved that all hypertrees are cop-win hypergraphs. In Theorems \ref{prodhgthm} and \ref{prismthm}, we determined the exact cop number of the Cartesian product of trees with cop-win hypergraphs, and of prism hypergraphs which have a Cartesian product-like structure. 

We have seen that the analysis of the game played on hypergraphs, compared to graphs, is significantly more challenging, and considering the $2$-section of hypergraphs is a useful tactic in such analysis. Some immediate open questions that we may ask, in view of the results of this paper, are the following. Firstly, are there other interesting classes of hypergraphs for which we can determine their cop-numbers? Secondly, can we determine the cop-numbers of some other Cartesian products of hypergraphs? In particular, note that in Corollary \ref{prodhgcor}, we have determined the cop number for the Cartesian product of hypertrees with anti-rank at least $3$. But what happens if we drop the anti-rank condition? Finally, it will be interesting to investigate the cop-numbers of other types of product hypergraphs.

We hope that the topic of the game of Cops and Robber played on hypergraphs will gain more significant attention in the near future.

\section*{Acknowledgements}
Pinkaew Siriwong is supported by Graduate School and Faculty of Science, Chulalongkorn University,\, and\, Science\, Achievement\, Scholarship\, of\, Thailand.\, Henry\, Liu\, is\, partially supported by the Startup Fund of One Hundred Talent Program of SYSU, and National Natural Science Foundation of China (No. 11931002). 

Pinkaew\, Siriwong\, acknowledges\, the\, generous\, hospitality\, of\, Sun\, Yat-sen\, University, Guangzhou, China. She was able to carry out part of this research with Henry Liu during her visit there from November 2019 to January 2020.

\end{document}